\documentclass[12pt]{amsart}
\usepackage{amssymb}

\usepackage{enumerate}

\usepackage{graphicx}

\makeatletter
\@namedef{subjclassname@2010}{%
  \textup{2010} Mathematics Subject Classification}
\makeatother

\newtheorem{thm}{Theorem}[section]
\newtheorem{cor}[thm]{Corollary}
\newtheorem{lem}[thm]{Lemma}
\newtheorem{prop}[thm]{Proposition}

\newtheorem{corol}{Corollary}
\newtheorem*{theorem-non}{Main Theorem}

\theoremstyle{definition}
\newtheorem{defin}[thm]{Definition}
\newtheorem{rem}[thm]{Remark}

\numberwithin{equation}{section}

\begin{document}

%%%%% To ease editing, for IMPAN journals add:

\baselineskip=17pt

%%%%%%%%%%%

%% In the running head, replace first names by initials 
%% and give an abbreviation of the title.

\title[Simplicity over Singular Hyperbolicity]{Simplicity over Singular Hyperbolicity}

\author[M. Fanaee]{M. Fanaee}
\address{Instituto de Matematica e Estatistica\\Universidade Federal Fluminense\\24.020-140 Niteroi, RJ - Brazil}
\email{mf@id.uff.br}

\date{}

\begin{abstract}
We prove that, when $r+\rho>0$, for a majory set (containing an open and dense subset) of $C^{r,\rho}$ fiber bunched linear cocycles over singular hyperbolic attractors, the Lyapunov exponents have multiplicity 1.
\end{abstract}



. 

\maketitle

\section{Introduction}

As a weaker form of uniform hyperbolicity, the class of singular hyperbolic systems is a vast familly of flows that contains the Axiom A systems [S67], 
the Lorenz flows ]L63] and the singular horseshoes [LP86], among other systems. 

More precisely, suppose that $M$ is a $C^1$ closed manifold. Fix some smooth Riemannian structure on M 
and an induced normalized volume form $m$ (called Lebesgue measure).
We denote by $\mathcal C^1(M)$ the space of all $C^1$ vector fields on $M$. Given a vector field $X\in\mathcal C^1(M)$  one obtains, 
by integration, a one parameter family of $C^1$ diffeomorphisms
$\{X^t:M\rightarrow M,~t\in\mathbb R\}$ which is called a flow on $M$ which satisfies
(i) $X^0\equiv id$, and (ii) $X^t \circ X^s=X^{t+s}$, for any $t,s\in\mathbb R$.

An invariant set $\Lambda\subset M$ is called an attractor of a vector field $X$ if there is a neighborhood $U$ of $\Lambda$ such that
$$\Lambda=\bigcap_{t>0}X^t(U)$$
(a repeller is an attractor for the reversed vector ﬁeld $-X$). An attractor $\Lambda$ is transitive if there exists a dense orbit in $\Lambda$, and it is $C^1$-robustly transitive if there exists a $C^1$-neghborhood $\mathcal N$ of the vector field $X$ such that the set
$\bigcap_{t>}Y^t(U)$
is transitive for any $Y\in\mathcal N$.

It is proved in [MPP04] that, for $C^1$-robustly transitive sets with singularities on closed 3-manifolds,  
there are either proper attractors or proper repellers and
the eigenvalues at the singularities satisfy the following inequalities:
\begin{eqnarray}
\alpha_{ss}<\alpha_s<0<-\alpha_s<\alpha_u.
\end{eqnarray}

The presence of singularities prevents these attractors from being hyperbolic.

\begin{defin}
An attractor $\Lambda\subset M$ with a finite number of hyperbolic singularities is a singular hyperbolic attractor 
if the bundle over $\Lambda$ can be written as an ivariant continuous splitting 
\begin{eqnarray}
T_\Lambda M=E^s\oplus E^{cu}
\end{eqnarray}
such that, with respect to an adapted metric, there exists some $0<\theta<1$ for which\\
1. the splitting is dominated: $||DX^t|E^s||.||DX^t|E^{cu}||\leq\theta^t$, for any $t>0$,\\
2. $E^s$ is contracting: $||DX^t|E^s||<\theta^t$, for any $t>0$,\\
3. $E^{cu}$ is volume expanding: $|det(DX^t|E^{cu})|\geq exp(-\theta t)$, for any $t\geq0$.
\end{defin}

From a measure theoretic view point, see [APPV09] for instance, there exists a unique invariant physical probability measure $\mu$ suprted on $\Lambda$ which is hyperbolic, meaning that at almost every point $x\in M$ there is an invariant splitting of the tangent bundle of the form 
$$T_xM=E^s(x)\oplus E^X(x)\oplus E^u(x)$$
where, with respect to dynamical cocycle $DX^t$, the stable sub-bundle $E^s(x)$ corresponds to negative Lyapunov exponent 
$$\lim_{|t|\rightarrow+\infty}\frac{1}{t}\log||DX^t|_{E^s(x)}||<0.$$
$E^X(x)$ is the one-dimensional direction of the flow corresponding to zero exponent
and $E^u(x)$ is the one-dimensional sub-bundle of vectors with positive Lyapunov exponent:
$\lim_{|t|\rightarrow+\infty}\frac{1}{t}\log||DX^t|_{E^u(x)}||>0$.

In a more general case, let $\pi:\mathcal V\rightarrow M$ be a measurable $d$-dimensional vector bundle over $M$. A linear cocycle over $X^t$ is a flow
$$F_A^t:\mathcal V\rightarrow\mathcal V$$
which acts by linear isomorphisms $A^t(x):\mathcal V_x\rightarrow\mathcal V_{f(x)}$ on the fibers. 
Oseledets Theorem [O69] states that, under a certain integrability condition and with respect to any invariant probability measure,
there exist a Lyapunov splitting of the vector bundle as
$$\mathcal V_x=E^1(x)\oplus...\oplus E^k(x),~1\leq k=k(x)\leq d$$
and real numbers $\lambda_1(x)>...>\lambda_k(x)$ called Lyapunov exponents defined by
$$\lambda_i(x)=\lim_{|t|\rightarrow+\infty}\frac{1}{t}\log||A^t(x).v||,~v\in E^i(x)\backslash\{0\},~1\leq i\leq k,$$
at almost every point. The Lyapunov splitting and the Lyapunov exponents are invariant by the flow $X^t$ and vary measurably with the base point $x$.
Then, for ergodic flows, the Lyapunov splitting and Lyapunov exponents do not depend on the base points and so are global properties of the system.\\

One problem is then to chracterize when all Lyapunov exponents have multiplicity 1.
This kind of problem arised by Furstunberg [F63] for cocycles over Bernoulli shifts when the cocycle depends only on the first coordinate.

Ledrappier [L86] proposed another approach to this problem and,
Viana, Gomez-mont, Bonatti, Avila and Santamaria 
(see for instance [BGV03], [BV04], [AV07], [ASV13]) improved it for H\"older continuous cocycles over chaotic maps (hyperbolic an partially hyperbolic maps).
In recent works of [F13] and [BV] there are some ideas to extend the last results for cocycles over flows. 
Here, we extend this type of criterion, in partucalr, for cocycles over singular hyperbolic flows.\\

\textit{Acknowledgement.} This work is supported by a CNPq-Brazil Post doctorate grant and is done at University of Porto-Portugal.

\section{The main setting and resualts}

The $C^{r,\rho}$ topology is defined by
$$||A^t||_{r,\rho}=\max_{0\leq i\leq r}\sup_x||D^iA^t(x)||+\sup_{x\neq y}\dfrac{||D^rA^t(x)-D^rA^t(y)||}{\mathrm{d}(x,y)^{\rho}}$$
(for $\rho=0$ omit the last term). 
We denote by $\mathcal G^{r,\rho}(M,d,\mathbb C)$ the Banach space of all linear cocycles $F^t_A$ for which $||A^t||_{r,\rho}<+\infty$,
for all $t\in\mathbb R$.

In this course, we assume $r+\rho>0$ which implies $\eta-$H\"older continuity:
$$\parallel A^t(x)-A^t(y)\parallel\leq\parallel A^t\parallel_{0,\eta}\mathrm{d}(x,y)^\eta,$$
where 
\[\eta=\left\{\begin{array}{cl}
\rho&r=0\\1&r\geq 1.
\end{array}\right.\]

\subsection{Fiber bunched cocycles}

Let $F_A^t$ be an $\eta$-H\"older connuous linetr cocycle.

\begin{defin}
We say that $F^t_A$ is fiber bunched if there exists some $0<\gamma<0$ such that 
\begin{eqnarray}
||A^t(x)||||A^t(x)^{-1}||\theta^{t\eta}<\gamma^t
\end{eqnarray}
for every $x\in M$ and all $t>0$.
\end{defin}

\begin{rem}
Fiber bunching is a $C^0$-open condition: if $F_A^t$ is fiber bunched then any sufficiently $C^0$-close cocycl0e is also fiber bunched, by definition. 
Hence, the set of all fibr bunched linear cocycles is an open set in the space of all linear cocycles. 
\end{rem}

\subsection{Simple Lyapunov exponents}

The main result of this work is the following.

\begin{theorem-non}
For any $r,\rho$ with $r+\rho>0$, the set of all $C^{r,\rho}$ fiber bunched linear cocycles over a singular hyperbolic flow 
for which all Lyapunov exponents have multiplicity 1, contains an open and dense subset of all fiber bunched linear cocycles in
$\mathcal G^{r,\rho}(M,d,\mathbb C)$, with respect to the invariant hyperbolic probability measure $\mu$. 
\end{theorem-non}

Singular hyperbolic attractors are motivated by the classical construction of so called Lorenz attractors: 
an invariant non-hyperbolic set of a robust flow given by the solutions of the Lorenz equations
\begin{eqnarray}
\begin{array}{l} \dot x=10(y-x), \\
\dot y=28x-y-xz, \\
\dot z=xy-\frac{8}{3}x.
\end{array}
\end{eqnarray}

To study the dynamical behavior of Lorenz attractors, in the late-seventies, was introduced a geometric model of the Lorenz attractor  
in [ABS77], [W79] and [GW79]. The next corollary is then straghtforward as a particular case of the Main Theorem.

\begin{corol}
For any $r,\rho$ with $r+\rho>0$ and with respect to $\mu$, the set of all $C^{r,\rho}$ fiber bunched linear cocycles over a geometric Lorenz attractor 
for which all Lyapunov exponents have multiplicity 1 
contains an open and dense subset of all fiber bunched linear cocycles in $\mathcal G^{r,\rho}(M,d,\mathbb C)$.
\end{corol}

It was proved in [T98] that the solutions of (1.4)  
perform as same as the geometric model of Lorenz attractor. Hence, we have immediately the following.

\begin{corol}
For any $r,\rho$ with $r+\rho>0$ and with respect to $\mu$, the set of $C^{r,\rho}$ linear cocycles over a Lorenz flow
for which all Lyapunov exponents have multiplicity 1 contains an open and dense subset of all fiber bunched linear cocycles in 
$\mathcal G^{r,\rho}(M,d,\mathbb C)$.
\end{corol}

We stress that the set of exceptional cocycles in Main Theorem and corollaries 1 and 2 is very meager and indeed corresponds to infinite codimension: it is contained in a finite union of closed submanifolds with arbitrary high codimension.

\begin{rem}
The regularity hypothesis $r+\rho>0$ in our statements is
necessary: it is proved by Bochi and Viana [B02,BV05] that for generic $C^0$ cocycles over
general transformations Lyapunov exponents ofen vanish. Also, for generic
$L^p$ cocycles, $0<p<+\infty$, the Lyapunov exponents always vanish (see [AB03] and [AC97]).
\end{rem}

\section{Cocycles over Pincar\'e maps}

Assume that $\pi:\mathcal V\rightarrow\Sigma$ be a measurable vector bundle over a measurable space $\Sigma$. 
A linear cocycle over a measurable transformation $f:\Sigma\rightarrow\Sigma$ is a transformation $F:\mathcal V\rightarrow\mathcal V$ 
satisfying $f\circ \pi=\pi\circ F$. Locally, a linear cocycle over $f$ is a transformation $$F_A:\Sigma\times\mathbb C^d\rightarrow \Sigma\times\mathbb C^d$$ 
which acts by linear isomorphisms $A(x)$ on fibers and has the form
$F_A(x,v)=(f(x),A(x)v).$
By definition $F_A^n(x,v)=(f^n(x),A^n(x)v)$, where
$$A^n(x)=A(f^{n-1}(x))~...~A(f(x))A(x),$$
for any $n\geq1$, and we define $A^0(x)=\mathrm{id}$.  Note that, conversely, any map $A:N\rightarrow GL(d,\mathbb C)$ defines a unique corresponding linear cocycle $F_A$ over $f$.

Let $\mu_f$ be a probability measure invariant by $f$, and assume that the map $x\mapsto\max\{0,\log\parallel A(x)\parallel\}$ is $\mu$-integrable.
Oseledets Theorem [O68] states that there exist an invariant filtration of $\mathbb C^d$ as
\begin{eqnarray}
\mathbb C^d=\mathcal V_0(x)>...>\mathcal V_k(x)=\{0\},~1\leq k=k(x)\leq d,
\end{eqnarray}
and corresponding invariant numbers (Lyapunov exponents) $\lambda_1(x)>...>\lambda_k(x)$, defined as
$$\lambda_i(x)=\lim_{n\rightarrow+\infty}\frac{1}{n}\log\parallel A^n(x)v_i\parallel,~v_i\in\mathcal V_{i-1}(x)\backslash\mathcal V_i(x),~1\leq i\leq k,$$
at almost every point. 
Lyapunov exponents and Lyapunov filtration are uniquely defined at almost every point and vary measurably with the base point $x$.
By invariance, we conclude that the Lyapunov exponents are constant if the invariant probability measure $\mu_f$ is ergodic.

\subsection{Poincar\'e maps}
Hereafter, we take $\Lambda$ to be a singular hyperbolic attractor of a vector field $X\in\mathcal C^1(M)$ where $M$ is a $C^1$ closed 3-manifold.
We assume that the splitting $E^s\oplus E^{cu}$ is extended to a neighborhood $U$ of $\Lambda$. We begin with existence of local stable and local unstable manifolds.

Given any $\epsilon>0$, take $I_\epsilon=(-\epsilon,\epsilon)$.
Let $\mathcal E^1(I_1,M)$ be the space of all $C^1$-embedding maps $h:I_1\rightarrow M$ 
and $\mathcal E^1(I_1\times I_1,M)$ be the space of all $C^1$-embedding maps $h:I_1\times I_1\rightarrow M$.

\begin{prop}[Existence of local stable and local center-unstable manifolds {[APPV09]}]
There exist continuous maps $h^s:U\rightarrow\mathcal E^1(I_1,M)$ and $h^{cu}:U\rightarrow\mathcal E^1(I_1\times I_1,M)$
such that, for any $\epsilon>0$ and $x\in U$, the local stable manifold $W^s_\epsilon(x)=h^s(x)(I_\epsilon)$ 
and local center-unstabel manifold $W^{cu}(x)=h^{cu}(x)(I_\epsilon\times\epsilon)$ exist.
Moreover, we have $TW^s_\epsilon(x)=E^s_x$ and $TW^{cu}_\epsilon(x)=E^{cu}_x$.
\end{prop}

We take a cross-section $\Sigma\subset U$ to be any $C^2$ embedded compact disc (diffeomorphic to $[-1,1]\times[-1,1]$) transverse to the $X$ at every point $x\in\Sigma$. As a direct consequence of the Implicit Function Theorem, we define the Poincar\'e map
\begin{eqnarray}
f:\Sigma\rightarrow\Sigma^\prime
\end{eqnarray}
on $\Sigma$ into another cross-section $\Sigma^\prime$ to be of the form $P(x)=X^{\tau(x)}(x)$ and define the Poincar\`e time function $\tau:\Sigma\rightarrow\mathbb R$ for which $X^{\tau(x})(x)\in\Sigma^\prime$. Note that $f$ and $\tau$ need not correspond to the first time that the orbits of $\Sigma$ encounter $\Sigma^\prime$.

\subsubsection{Hyperbolicity of Poincar\'e maps}

Without loos of generallity, we can assume that $\Sigma=\Sigma^\prime$.
The continuity of the invariant splitting $E^s\oplus E^{cu}$  in (1.2) over $U$ induces a continuous splitting
$$E^s_\Sigma\oplus E^{u}_\Sigma$$
of the tangent bundle $T\Sigma$ of $\Sigma$ defined by
\begin{eqnarray}E^s_\Sigma(x)=E^s(x)\cap T_x\Sigma,~~~E^u_\Sigma(x)=E^{cu}(x)\cap T_x\Sigma.
\end{eqnarray}

\begin{prop}[{[APPV09]}]
Suppose that $f$ is a Poincar\'e map on a cross-section $\Sigma$. Then, for every $x\in\Sigma$, we have
\begin{eqnarray}
Df(x)(E^s_\Sigma(x))=E^s_\Sigma(f(x)), ~Df(x)(E^u_\Sigma(x))=E^u_\Sigma(f(x)).
\end{eqnarray}
Moreover, there exists a time $t_0>0$ such that if $\tau(.)>t_0$ then
\begin{eqnarray}
||Df|E^s_\Sigma||<\theta,~||Df|E^u_\Sigma||>\theta^{-1},
\end{eqnarray}
($0<\theta<1$ as in definition 1.1).
\end{prop}

We exhibit the stable and unstable manifolds for a Poincar\'e map $f:\Sigma\rightarrow\Sigma$
to be the natural candidates $W^s_\Sigma(x)=W^s_\epsilon(x)\cap\Sigma$ and $W^u_\Sigma(x)=W^{cu}_\epsilon(x)\cap\Sigma$.
We also can take adapted cross-sections for which theses stable and unstable manifolds are invariant: 
$f(W^s_\Sigma(x))\subset W^s_\Sigma(f(x))$ and $f(W^u_\Sigma(x))\subset W^{u}_\Sigma(f(x))$.

\subsection{A global Poincar\'e map}

Let $\Lambda$ be a singular hyperbolic set for a vector field $X\in\mathcal C^1(M)$. 
We assume that the set $Sing(\Lambda)$ of the singularities of $\Lambda$ is not empty (otherwise $\Lambda$ is hyperbolic).

We take an adapted linearizing neighborhood of a normalized singularity $\sigma=(0,0,0)$ (and then for all singularities), for which\\
1. According to the eigenvalues of hyperbolic singularities in (1.1)  the vector field has the form
$$X(x,y,z)=(\alpha_{ss}x,\alpha_uy+o(x,y,z),\alpha_sz+o(x,z))$$
where $o$ denotes the corresponding higher order terms. Hence the flow is locally given by
$$X^t(x,y,z)=(x\exp(t\alpha_{ss}),y\exp(t\alpha_u)+o(x,y,z),z\exp(t\alpha_s)+o(x,z)).$$
2. We can take $W^s_{loc}(\sigma)=\{y=0\},~W^{ss}_{loc}(\sigma)=\{y=z=0\}$ and $W^u_{loc}(\sigma)=\{x=z=0\}$. \\
3. The planes $z=1$ and $z=-1$ are transversal to the flow where the vector field points inward the region containing the singularity. Then, we can find two rectangles
$$\Sigma^+_\sigma=\{(x,y,1)|~-1\leq x\leq 1,~1\leq y\leq 1\}\subset\{z=1\}$$
and
$$\Sigma^-_\sigma=\{(x,y,-1)|~-1\leq x\leq 1
,~-1\leq y\leq 1\}\subset\{z=-1\}$$
4. There exist two disks $\Delta^+_\sigma\subset \{y=1\}$ and $\Delta^-_\sigma\subset\{y=-1\}$ that for any point $x\in\Sigma_+\cup\Sigma_-$ there is $t>0$ such that $X^t(x)\in\Delta_+\cup\Delta_-$.\\
5. We set $\Gamma_\sigma=\Sigma_\sigma\cap W^s_{loc}(\sigma)$. Then $\Gamma_\sigma$ divides $\Sigma_\sigma$ into two semi boxes.\\

image 1\\

We denote by $\Sigma=\bigcup_{\sigma\in Sing(\Lambda)}\Sigma_\sigma$ a system of transversal sections and $\Gamma=\bigcup_{\sigma\in Sing(\Lambda)}\Gamma_\sigma$. In this way, we have a global Poincar\'e map
\begin{eqnarray}
f:\Sigma\backslash\Gamma\rightarrow\Sigma.
\end{eqnarray}
defined by $f(x)=X^{\tau(x)}(x)$ where $\tau:\Sigma\backslash\Gamma\rightarrow[0,+\infty)$ is the return time Poincar\'e function (set $\tau|_{\Lambda}\equiv+\infty$).

\subsection{Hyperbolic invariant measures}

From now on, we assume that $X^t$ is a $C^2$ flow on $M$. 
It is well known [PT93], under this assumption, that the stable leaf $W^s_\Sigma(x)$, for every $x\in\Sigma$,
is a $C^2$ embedded disk and this leaves define a $C^{1+\rho}$ foliation $\mathcal F_\sigma$ of each $\Sigma_\sigma\in\Sigma$.
The canonical projection $\pi_s:\Sigma\rightarrow\mathcal F$ 
which assigns to any $x\in\Sigma$ the atom $\xi\in\mathcal F$ for which $x\in\xi$,
induces a one dimentional map $h:\mathcal F\backslash\Gamma\rightarrow\mathcal F$ defined by
$$h\circ\pi_s=\pi_s\circ f.$$
$h$ is a $C^{1+\rho}$ piece-wise expanding map. It is clear that$\pi_s$ induces measurable and topologic structures of $\Sigma$ to $\mathcal F$.

There exists an absoloutly continuous (with respect to Lebesgue measure) probability measure $\mu_h$ on $\mathcal F$ invariant by $h$ 
(see for instance [V97]). Then, the probability measure $\mu_f$ defined on $\Sigma $ by
\begin{eqnarray}
\int\eta d\mu_f=\lim_{n\rightarrow+\infty}\int\inf_{x\in\xi}(\eta\circ f^n)d\mu_h=\lim_{n\rightarrow+\infty}\int\sup_{x\in\xi}(\eta\circ f^n)d\mu_h,
\end{eqnarray}
for any continuous function $\eta:\Sigma\rightarrow\mathbb R$, is invariant by $f$. More over, $\mu_f$ is ergodic if $\mu_h$ is ergodic [APPV09].

An invariant probability measure is hyperbolic if all Lyapunov exponents with respect to the dynamic cocycle are non-zero.

\begin{lem}
The unique invariant probability measure $\mu_f$ is hyperbolic.
\end{lem}
\begin{proof}
The hyperbolicity of the invariant measure $\mu_f$ is an immediate consequence of hyperbolicity of the Poincar\'e map $f$ stated in Proposition 3.2.
\end{proof}

%We remamber that a probability measure $\mu_f$ has local product structure if 
%$$\mu_f=\beta(x)(\mu_f^s\times\mu_f^u)$$
%where $\mu_f^s$ and $\mu_f^u$ are the corresponding projections of  $\mu_f$ to the local stable and local unstable sets, 
%and $\beta:\Sigma\rightarrow[0,+\infty)$ is a measurable function (see [V08] for more details).

%We recall that, in a differentiable setting, the hyperbolicity of invariant measure implies local product structure of measure 
%(see [B08], [V08] and [L87]  for the proofs and more details). We have then the following, by the last proposition.

%\begin{cor}
%$\mu_f$ has local product structure. 
%\end{cor}

\subsection{Markov structure for global Poincar\'e maps}

There exists a continuous extension to the closur, of the map $h:\mathcal F\backslash\Gamma\rightarrow\mathcal F$ as a map on intervals of $\mathbb R$. 
Consequently the restriction of $f:\Sigma\backslash\Gamma\rightarrow\Sigma$ to each of the connected components of $\Sigma\backslash\Gamma$ 
admits a continuous extention to the closur, each one collapsing $\Gamma$ to a single point in $\Sigma$. 
For natural convenience and now after, we take $f:\Sigma\rightarrow\Sigma$
to be the global Poincar\'e map extended to a $2$-valued map defined on the whole cross-section $\Sigma$ 
and continuous on each of the connected components of $\Sigma$. 
In the rest of this course, we will consider global Poincar\'e maps as discussed here.

We call stable bundary of $\Sigma_\sigma$ the image of $[-1,1]\times\{-1,1\}$, for some $\Sigma_\sigma$.

\begin{defin}
A connected subset $B$ in $\Sigma$ is a band
if it intersects both connected components of the stable boundary of $\Sigma$ and $B\cap\Lambda\neq\emptyset$. 
\end{defin}

The next theorem guarantees existence of a Markovian structure for Poicar\'e maps.
Assume that $\pi_s$ is the projection map on $\Sigma$ along stable manifolds.

\begin{thm}[{[AP07]}]
There exist a system of transversal sections $\Sigma$ such that for any band $B\subset\Sigma$ there is a sub-band $\hat B\subset B$ 
and a global Poincar\'e map $f:\Sigma\rightarrow\Sigma$ such that $f(\hat B)$ covers some $\Sigma_\sigma$. Moreover, 
$$\pi_s(f(\hat B))=\pi_s(\Sigma_\sigma).$$
\end{thm}

image 2\\

We take the maximal invariant set $S=\bigcap_{n\geq0}f^n(\Sigma)$. 
By construction and definition of $\mu_f$ in (3.7), we conclude that $S\subset supp(\mu_f)$. 
Then, the next corollary is immediate and states that there exists a global Poincar\'e map which is a Markov map.

\begin{cor}
There exist a global Poincar\'e map $f$ and a partition of $S$ by a collection $\{B(i)\cap S:~i\in\mathbb N\}$ of bands intersecting $S$,
such that the transformation $g=f|_S$ is a return Markov map to $S$.
\end{cor}

\subsection{Fiber bunched cocycles}

Assume $0<\theta<1$ the same as in Definition 1.1 and Proposition 3.2.

\begin{defin}
We say that $F_A$ is fiber bunched if there exists some $0<\gamma<0$ such that 
\begin{eqnarray}
||A(x)||~||A(x)^{-1}||~\theta(x)^\eta<\gamma
\end{eqnarray}
for every $x\in\Sigma$ and $0<\theta(x)<\theta$.
\end{defin} 
\begin{rem}
As we mentioned in Reamrk 1.3, if $F_A$ is fiber bunched then there exists a $C^0$-neighborhood of $F_A$ for which any cocycle is fiber bunched
in this neighborhood.
Hence, the set of all fibr bunched linear cocycles is a $C^0$-open set $\mathcal U$ in the space of all $\eta$-H\"older continnuous cocycles. 
\end{rem}

\subsection{Simple Lyapunov spectrum}

Avila and Viana [AV07] extended, a simplicity criterion stated already by Bonatti and Viana [BV04]
for cocycles over shift maps, to cocycles over Markov maps.

\begin{thm}[{[AV07],[F]}]
For any $\eta>0$,o and any hyperbolic invariant probability measure, 
the set of all $C^\eta$ fiber bunched linear cocycles over any Markov map for which all Lyapunov exponents have multiplicity 1,
contains an open and dense subset of $\mathcal U$. 
Moreover, the exceptional set of cocycles has infinite codimension: 
it is contained in a union of closed submanifolds with arbitrary high codimention.
\end{thm}

Therefore, we can conclude this section with the following.

\begin{cor}
There exists a global Poincar\'e map $f:\Sigma\rightarrow\Sigma$ such that, for any $r+\rho>0$, and with respect to $\mu_f$, 
the set of $C^{r,\rho}$ fiber bunched linear cocycles over $f$ for which all Lyapunov exponents have multiplicity 1, 
contains an open and dense subset of $\mathcal U$.
\end{cor}

\begin{proof}
We take the global Poincar\'e map as in Corollary 3.7.which is a Markov map.
By Lemma 3.3, the invariant probability measure $\mu_f$ is hyperbolic.
The proof is now done by the last theorem.
\end{proof}

\section{Cocycles over the suspension flow}

Nowafter, we assume that $f:\Sigma\rightarrow\Sigma$ is the global Poincar\'e map in Corollary 3.14 and $\tau(.)$ the corresponding measurable Poincar\'e time function such that $\inf(\tau)>0$.

To construct a suspension flow from $f$, we define an equivalence relation $\sim$ on $\Sigma\times [0,+\infty)$ generated by $(x,\tau(x))\sim(f(x),0)$.

Let $V=\Sigma\times[0,+\infty)/\sim$ be the qucient space with respect to $\sim$. Then the corresponding canonical projection $\pi:\Sigma\rightarrow V$ induces a topology and a Borel $\sigma$-algebra of measurable sets on $V$. It is easy to check that
\begin{eqnarray}
X_f^t(\pi(x,s))=\pi(x,s+t), ~(x,s)\in\Sigma\times[0,+\infty),~t>0,
\end{eqnarray} 
defines a semi-flow $X_f^t$ on $V$.

By suspension tools on equivalence relation defined by $\sim$, 
we can construct a probability measure, using the product of $\mu_f$ by 1-dimentional lebesgue measure, on $V$.
This implies existence of a unique probability measure $\mu_X$ on $V$ which is invariant by the semi-flow $X_t$.
Moreover, $\mu_X$ is ergodic (see [APPV09] for more details).

Suppose that $F_A^t$ is an $\eta$-H\"older continuous linear cocycle over the suspension flow $X_f^t$. We define a corresponding $\eta$-H\"older continuous linear cocycle over a global Poincar\'e map $f:\Sigma\rightarrow\Sigma$ defined by $A_f:\Sigma\rightarrow GL(d,\mathbb C)$ as
\begin{eqnarray}
A_f(x)=A^{\tau(x)}(x),~x\in\Sigma.
\end{eqnarray}

\begin{lem}
 Assume that $\tau(.)>1$. If $F^t_A$ is a fiber bunched linear cocycle over semi-flow $X^t_f$ 
 then the corresponding linear cocycle defined by $A_f$ over global Poincar\'e map $f$ is fiber bunched.
\end{lem}

\begin{proof}
suppose that $F^t_A$ is a fiber bunched cocycles over the semi-flow $X^t_f$, i.e, for some $0<\gamma<1$ and any $x\in V$,
$$||A^t(x)||||A^t(x)^{-1}||\theta^{t\eta}<\gamma^t,$$
for all $t\in\mathbb R$. Then, in particular,  for $t=\tau(x)$, we have 
$$||A^{\tau(x)}(x)||||A^{\tau(x)}(x)^{-1}||\theta^{\tau(x)\eta}<\gamma^{\tau(x)}.$$
We assume that $\tau(.)>1$ which implies that $\theta(x)=\theta^{\tau(x)}<\theta$, for all $x\in\Sigma$, and $\gamma^{\tau(x)}<\gamma$. Hence, we have
$$||A_f(x)||||A_f(x)^{-1}||\theta(x)^\eta<\gamma,$$
for every $x\in\Sigma$.
\end{proof}

\begin{lem}
Lyapunov exponents with respect to $A^t$ have multiplicity 1 if and only if Lyapunov esxponents with respect to $A_f$ have multiplicity 1.
\end{lem}

\begin{proof}
The Lyapunov exponents of $A_f$ are obtained by multiplying those of $A^t$ by the average return time
\begin{eqnarray}
s_n(x)=\sum_{j=0}^{n-1}\tau(\hat f^j(x)),~x\in\Sigma.
\end{eqnarray}
Given any non-zero vector $v$, we have
\begin{eqnarray}
\lim_{n\rightarrow+\infty}\frac{1}{n}\log||A_f^n(x)v||=\lim_{n\rightarrow+\infty}\frac{1}{n}\log||A^{s_n(x)}(x)v||.
\end{eqnarray}
But, for $\mu$-almost every ponit $x\in\Sigma$, (4.3) is equal to
$$\lim_{n\rightarrow+\infty}\frac{1}{n}s_n(x)\lim_{m\rightarrow+\infty}\frac{1}{m}\log||A^m(x)v||.$$
As $\frac{1}{n}s_n(x)$ converges to $\mu_f(\tau)<+\infty$, the proof is complete.
\end{proof}

\begin{prop}
Suppose that $F_A^t$ is a fiber bunched linear cocycle over the suspension flow $X^t_f$. Then the map
\begin{eqnarray}
A^t\mapsto A_f\in\mathcal U\subset\mathcal G^{r,\rho}(\Sigma,d,\mathbb C)
\end{eqnarray}
is a submersion.
\end{prop}

\begin{proof}
By definition, for any tangent vector $B^t$ in the tangent space of a linear cocycle $A^t\in\mathcal C^{r,\rho}(V,d,\mathbb C)$, we have
\begin{eqnarray}
(\partial_{A^t})(A_f)(B^t)=B_f.
\end{eqnarray}
The derivative in (4.6) is surjective. 

For any $B\in\mathcal G^{r,\rho}(\Sigma,d,\mathbb C)$, we consider the natural suspension $B^t$ of $B$ generated by
\begin{eqnarray}
B^t(x)=(B(x),t),~0\leq t<\tau(x),
\end{eqnarray}
By definition, $B^t$ is an $\eta$-H\"older linear cocycle over the semi-flow $X_f^t$. Then, the derivative is surjective.
\end{proof}
\begin{cor}
There exists a semi-flow $X^t_f$ such that, for any $r+\rho>0$ and with respect to $\mu_X$,
the set of $C^{r,\rho}$ fiber bunched linear cocycles over $X^t_f$ for which all Lyapunov exponents have multiplicity 1, 
contains an open and dense subset of $C^{r,\rho}$-open set $\mathcal V$ of all fiber bunched linear cocycles in $\mathcal G^{r,\rho}(V,d,\mathbb C)$.
Even more, The complement set of cocycles corresponds to infinite codimention.
\end{cor}

\section{Cocycles over singular hyperbolic flows}

We define a map $\phi:\Sigma\times[0,+\infty)\rightarrow U$ by
\begin{eqnarray}
\phi(x,t)=X^t(x).
\end{eqnarray}
But, as $\phi(x,\tau(x))=\phi(f(x),0)$, we can deduce a natural qucient map $\Phi:V\rightarrow U$ for which
\begin{eqnarray}
X^t\circ\Phi=\Phi\circ X_f^t,
\end{eqnarray}
through the equivalence relation $\sim$ defined earlear. Then, $\Phi$ is a homeomorphism onto its image $\Phi(V)$ and, even more, it is a diffeomorphism on the full lebesgue measure set $V\backslash\pi(\Gamma)$.

We define the basin set of a probability measure $\mu$ to be the set $B(\mu)$ of points $x\in M$  for which
$$\lim_{t\rightarrow+\infty}\frac{1}{t}\int_0^t\phi(X^t(x))dt=\int\phi d\mu,$$
for any measurable function $\phi:M\rightarrow\mathbb R$. An invariant probability measure $\mu$ is then a physical measure for the flow $X^t$ if the basin set of $\mu$
has positive Lebesgue measure: $m(B(\mu))>0$. As $\Phi$ is a homeomorphism on its image, we can define a measure on $\Phi(v)$, using $\mu_X$, and then we have tha following.

\begin{thm}[{[APPV09]}]
There exists a unique invariant physical probability measure $\mu$ supported on $\Lambda$ which  is ergodic and hyperbolic.
\end{thm}

Finally, we extend the results of the last section with respect to the semi-flow $X^t_f$ to the original singular hyperbolic flow $X^t$.

Assume that $A^t:\Phi(V)\rightarrow SL(d,\mathbb C)$ is an $\eta$-H\"older continuous linear cocycle over $X^t|_{\Phi(V)}$. With respect to $A^t$ there is an $\eta$-H\"older linear cocycle over the semi-flow $X_f^t$ defined by
\begin{eqnarray}
A_\Phi^t(x)=A^t(\Phi(x)),~x\in V.
\end{eqnarray}
By definition, if $F^t_A$ is fiber bunched then the linear cocycle defined by $A_\Phi^t$ is fiber bunched.
As $\Phi$ is a homeomorphism onto its image $\Phi(V)$, we have immediately the following.
 
\begin{lem}
Lyapunov exponents of $A_\Phi^t$ have multiplicity 1 if and only if Lyapunov exponents of $A^t$ have multiplicity 1.
\end{lem}

\begin{prop}
Assume that $F^t_A$ is a $C^{r,\rho}$ fiber bunched linear cocycle over flow $X^t$. Then the map
\begin{eqnarray}
A^t\mapsto A^t_\Phi\in\mathcal V\subset\mathcal G^{r,\rho}(V,d,\mathbb C)
\end{eqnarray}
is a submersion.
\end{prop}
\begin{proof}
By definition, for all tangent vector$B^t$ in the tangent space of any $A^t\in\mathcal G^\eta(U,d,\mathbb C)$, we have
\begin{eqnarray}
(\partial_{A^t})(A_\Phi^t)(B^t)=B_\Phi^t.
\end{eqnarray}
To prove that (5.5) is surjective, suppose that $B^t\in\mathcal G^{r,rho}(V,d,\mathbb C)$. 
Then, we define a tangent vector $\mathcal B^t\in\mathcal G^{r,\rho}(U,d,\mathbb C)$ as
$$\mathcal B^t(x)=B^t(\Phi^{-1}(x)),~x\in\Phi(V).$$
Now, it is easy to see that $\mathcal B_\Phi^t(x)=B^t(x)$, for any $x\in V$.
\end{proof}
As $\Phi(V)$ do have full lebesgue measure in $U$, the proof of Main Theorem is now completed.\\

We conclude this work by recall the fact that Theorem 3.13 is valid for cocycles with values in $\mathrm{GL}(d,\mathbb R)$ 
(see [BV04] and [V08] for more details).
Hence, the Main Theorem and corollaries 1 and 2 are valid for real-valued cocycles over singular hyperbolic attractors.


\begin{thebibliography}{HD}

\bibitem[ABS77]{}
V. Afarimovich, V. Bykov and L. Shilnikov. On the apearence and structure of the Lorenz attractor. Doklady Akademii Sciences USSR 234 (1977) 336-339.

\bibitem[AP10]{}
V. Araujo and M. Pacifico. Three dimensional flows. Springer-Verlag (2010).

\bibitem[AB03]{}
A. Arbieto and J. Bochi. $L^p$- generic cocycles have one-point Lyapunov spectrum.
Stochastic Dynamics 3 (2003) 73–81.

\bibitem[AC97]{}
L. Arnold and N. Cong. On the simplicity of the Lyapunov spectrum of products of random matrices.
Ergodic Theory and Dynamacial Systems 17 (1997) 1005–1025.

\bibitem[AP07]{}
A. Arroyo and E. Pujals. Dynamical properties of singular-hyperbolic attractors.  Discrete and Contintinuous Dynnamical Systems  19 (2007) 67--87.

\bibitem[APPV09]{}
V. Araujo, M. Pacifico, E. Pujals and M. Viana. Singular-hyperbolic attractors. Transactions of the American Mathematical Society 361 (2009) 2431-2485.

\bibitem[ASV13]{}
A. Avila, J. Santamaria and M. Viana. Holonomy invariance: rough regularity and applications to Lyapunov exponents. Ast\'erisque 358 (2013) 13-74.

\bibitem[AV07]{}
A. Avila and M. Viana. Simplicity of Lyapunov spectra: a sufficient condition.  Potugaliae Matematica 64 (2007) 311-376.

\bibitem[BV]{}
M. Bessa and P. Varandas. Positive Lyapunov exponents for symplectic cocycles. arxiv.org.

\bibitem[BV05]{}
J. Bochi and M. Viana. The Lyapunov exponents of generic volume preserving and symplectic systems. Annals of Mathematics 161 (2005) 1423-1485.

\bibitem[BDV04]{}
C. Bonatti, L. Diaz and M. Viana. Dynamics Beyond Uniform Hyperbolicity: A Global Geometric and Probabilistic Perspective. 
Encyclopaedia of Mathematical Sciences 102 Springer-Verlag (2004).

\bibitem[BGV03]{}
C. Bonatti, X. Gomez-Mont and M. Viana. G\'{e}n\'{e}ricit\'{e} d'exposants de Lyapunov non-nuls pour des produits d\'{e}terministes de matrices.
Annales de l'Institute Henri Poincar\'{e} 20 (2003) 579-624.

\bibitem[BV04]{}
C. Bonatti and M. Viana. Lyapunov exponents with multiplicity 1 for deterministic products of matrices.
Erodic Theory and Dynamical Systems 24 (2004) 1295-1330.

\bibitem[F13]{}
M. Fanaee. Linear cocycles over Lorenz-like flows. Publicaciones Matem\'aticas del Uruguay 14 (2013) 136-146.

\bibitem[F]{}
M. Fanaee. On simplle cocycles over Markov maps.

\bibitem[GW79]{}
J. Guckenheimer and R. Williams. Structural stability of Lorenz attractors. Publications Math\'{e}matiques de l'IH\'{E}S 50 (1979) 59-72.

\bibitem[K85]{}
G. Keller. Generalized bounded variation and applications to piecewise monotonic transformations. 
Z. Wahrscheinlichkeitstheorie verw. Gebiete 69 (1985) 461–478.

\bibitem[LP86]{}
R. Labarca and M. Pacifico. Stability of singular horseshoes. Topology 25 (1986) 337-352.

\bibitem[L86]{}
F. Ledrappier. Positivity of the exponent for stationary sequences of matrices, in Lyapunov
exponents (Bremen, 1984) Lectur Notes in Mathematics Springer 1186 (1986) 56–73.

\bibitem[L87]{}
F. Ledrappier. Dimension of invariant measures. Teubner-Texte Mathematics 94 (1987) 116–124.

\bibitem[L63]{}
E. Lorenz. Deterministic nonperiodic flow. Journal of the Atmospheric Sciences 20 (1963) 130-141.

\bibitem[MPP99]{}
C. Morales, M. Paciﬁco and E. Pujals. On singular hyperbolic systems. Proceedings of the American Mathematical Society 127 (1999) 3393–3400.

\bibitem[MPP04]{}
C. Morales, M. Pacifico and E. Pujals. Rubost transitive singular sets for 3-flows are partially hyperbolic attractors or repellers. Annals of Mathematics 160 (2004) 357-432.

\bibitem[O68]{}
V. Oseledets. A multiplicative ergodic theorem. Transactions of the Moscow Mathematical Society 19 (1968) 197-231.

\bibitem[PT93]
J. Palis and F. Takens. Hyperbolicity and sensitive-chaotic dynamics at homoclinic bifurcations. 
Cambridge University Press (1993).

\bibitem[S67]{}
S. Smale. Differentiable dynamical systems. Buletin of the Americam Mathematical Society 73 (1967) 747-817.

\bibitem[T98]{}
W. Tuker. The Lorenz attractor exists. University of Uppsala (1998).

\bibitem[V97]{}
M. Viana. Stochastic dynamics of deterministic systems. 21th Brazilian Mathematics Colloquium (1997).

\bibitem[V08]{}
M. Viana, Almost all cocycles over any hyperbolic system have non-vanishing Lyapunov exponents, Annals of Mathematics 167 (2008) 643-680.

\bibitem[W79]{}
R. Williams, The structure of the Lorenz attractor, Publications Math\'ematiques de l'IH\'{E}S 50 (1979) 73-99.\\\

\end{thebibliography}
\end{document}